\documentclass[pdftex]{amsart}

\address{Research and Education Center for Natural Sciences, Keio University, 4-1-1 Hiyoshi, Kohoku-ku, Yokohama, Kanagawa, 223-8521, Japan}
\email{\href{mailto:isoshima@keio.jp}{isoshima@keio.jp}}

\usepackage{amsmath,amssymb}
\usepackage{overpic}
\usepackage{amsfonts}
\usepackage{graphicx}

\usepackage{hyperref}
\usepackage{mathtools}

\usepackage[hang,small,bf]{caption}
\usepackage[subrefformat=parens]{subcaption}
\usepackage{xcolor}
\usepackage{amsthm}
\usepackage{longtable}
\usepackage{bm}

\theoremstyle{plain}
\newtheorem{thm}{Theorem}[section]
\newtheorem{prop}[thm]{Proposition}
\newtheorem{lem}[thm]{Lemma}
\newtheorem{cor}[thm]{Corollary}
\newtheorem*{thm*}{Theorem}
\newtheorem*{cor*}{Corollary}
\newtheorem*{prop*}{Proposition}

\theoremstyle{definition}
\newtheorem{dfn}[thm]{Definition}
\newtheorem{rem}[thm]{Remark}

\newtheorem{que}[thm]{Question}

\newtheorem*{que*}{Question}
\newtheorem*{con*}{Conjecture}
\newtheorem*{nota*}{Notation}

\begin{document}

\title[Minimal genus trisection diagrams of the elliptic surfaces $E(n)$]{Minimal genus trisection diagrams of the elliptic surfaces $E(n)$ via handle diagrams}

\author{Tsukasa Isoshima}

\begin{abstract}
Lambert-Cole and Meier showed that the elliptic surface $E(n)$ admits a $(12n-2,0)$-trisection, considering the property that $E(n)$ is a certain double branched cover of $S^2 \times S^2$, which is a minimal genus trisection. In this paper, we clarify a way to construct an explicit $(12n-2,0)$-trisection diagram of $E(n)$ from its handle diagram arising from its Lefschetz fibration. 
\end{abstract}

\maketitle

\section{Introduction}\label{sec:intro}

In 2016, Gay and Kirby \cite{MR3590351} introduced the notion of trisections into 4-dimensional topology. Roughly speaking, a \textit{trisection} of a closed 4-manifold is a decomposition of the closed 4-manifold into three 4-dimensional 1-handlebodies, each of which is diffeomorphic to a connected sum of several copies of $S^1 \times D^3$. The triple intersection of the three 4-dimensional 1-handlebodies is diffeomorphic to a closed surface of genus $g$, and this $g$ is called the \textit{genus} of the trisection.

Trisection diagrams are new diagrams describing 4-manifolds \cite{MR3590351}. A trisection diagram consists of an orientable closed surface and three cut systems on the surface called $\alpha$-, $\beta$- and $\gamma$-curves (red, blue and green, respectively; see Figure \ref{fig:CP^2} for example). It is known \cite{MR3590351} that two 4-manifolds are diffeomorphic if and only if two corresponding trisection diagrams are related by surface diffeomorphisms, handle slides among the same family curves and (de)stabilizations (Corollary \ref{cor:henkei}).

The minimum integer $g$ such that a 4-manifold admits a trisection of genus $g$ is called the \textit{trisection genus} of the 4-manifold. 
Trisection genera can be determined by combining upper and lower bounds. Upper bounds are often obtained either by constructing trisections or by constructing trisection diagrams. Examples of studies that obtain upper bounds via the former approach include \cite{MR3917737, MR3824315, MR4477412}. Examples obtained via the latter approach include \cite{MR3917737, MR4890820, MR4873874, takahashi2026trisectiongenusknottraces}.

Lambert-Cole and Meier \cite[Theorem 7.7]{MR4493476} showed by the former approach rather than the latter that the trisection genus of the elliptic surface $E(n)$ is $12n-2$, using the property that $E(n)$ is a certain double branched cover of $S^2 \times S^2$. 

In this paper, we clarify a way to construct an explicit minimal genus trisection diagram of $E(n)$ as the latter approach. We firstly construct an explicit large genus trisection diagram of $E(n)$, using an algorithm that takes handle diagrams to trisection diagrams introduced by Kepplinger \cite{MR4460230} (Theorem \ref{thm:algorithm}).

\begin{prop*}[Proposition \ref{prop:E(n)}]
Figure \ref{fig:step6} is a $(36n^2+6n+6;2,2,36n^2-6n+4)$-trisection diagram of $E(n)$ obtained by performing the algorithm in Theorem \ref{thm:algorithm} to the handle diagram in Figure \ref{fig:E(n)_Kirby}. 
\end{prop*}

Figure \ref{fig:E(n)_Kirby} is a handle diagram of $E(n)$ obtained from its Lefschetz fibration.
We then show the following main theorem.

\begin{thm*}[Theorem \ref{thm:main theorem}]
A $(12n-2,0)$-trisection diagram of $E(n)$ can be obtained from a $(36n^2+6n+6;2,2,36n^2-6n+4)$-trisection diagram of $E(n)$ in Figure \ref{fig:step6} by handle slides and destabilizations. 
\end{thm*}

We prove Theorem \ref{thm:main theorem} by destabilizing the trisection diagram in Figure \ref{fig:step6} explicitly and reducing its genus. In particular, we can draw a $(12n;0,2,0)$-trisection diagram of $E(n)$ as in Figure \ref{fig:after_destab4}. As a corollary, we recover the result of Lambert-Cole and Meier.

\begin{cor*}[Corollary \ref{cor:genus}]
The trisection genus of $E(n)$ is $12n-2$.
\end{cor*}

Two trisections of the same 4-manifold are \textit{isotopic} if there exists an ambient isotopy taking each 1-handlebody to the corresponding one (Definition \ref{def:isotopy}).

\begin{que*}[Question \ref{que:isotopic}]
Are the minimal genus trisection of $E(n)$ in Theorem \ref{thm:main theorem} and the one in \cite[Theorem 7.7]{MR4493476}, isotopic?
\end{que*}

\section*{Acknowledgement}
The author would like to thank Natsuya Takahashi for helpful comments on trisection genus. The author was partially supported by JSPS KAKENHI Grant Number JP25KJ0301.

\section{Preliminaries}
In this paper, unless otherwise stated, each 4-manifold is compact, connected, oriented and smooth. It is denoted by $X \cong Y$ that manifolds $X$ and $Y$ are diffeomorphic.

\subsection{Trisections of 4-manifolds}

In this subsection, we recall the definition and some properties for trisections.

\begin{dfn}
Let $X$ be a closed 4-manifold and $g$ and $k_i$ non-negative integers with $k_i \le g$ for $i=1,2,3$. A $(g;k_1,k_2,k_3)$-\textit{trisection} of $X$ is a 3-tuple $(X_1,X_2,X_3)$ satisfying the following conditions:
\begin{itemize}
\item $X=X_1 \cup X_2 \cup X_3$,
\item For each $i=1,2,3$, $X_i \cong \natural_{k_i} S^1 \times D^3$,
\item For each $i=1,2,3$, $X_i \cap X_{i+1} \cong \natural_{g} S^1 \times D^2$, where $X_4=X_1$ and 
\item $X_1 \cap X_2 \cap X_3 \cong \#_{g} S^1 \times S^1 = \Sigma_g$.
\end{itemize}
\end{dfn}

Let $H_{\alpha} = X_3 \cap X_1$, $H_{\beta} = X_1 \cap X_2$ and $H_{\gamma} = X_2 \cap X_3$. The union $H_{\alpha} \cup H_{\beta} \cup H_{\gamma}$ is called the \textit{spine} of the trisection. The integer $g$ is called the \textit{genus} of the trisection. The 4-tuple $(g;k_1,k_2,k_3)$ is called the \textit{type} of the trisection. If $k_1=k_2=k_3$, the trisection is said to be \textit{balanced}, and the type is denoted by $(g,k)$ for short, where $k=k_i$ ($i=1,2,3$). Otherwise, the trisection is said to be \textit{unbalanced}.

\begin{dfn}
The \textit{trisection genus} $g(X)$ of a closed 4-manifold $X$ is defined as the minimum integer $g$ such that $X$ admits a trisection of genus $g$. 
\end{dfn}

\begin{dfn}\label{def:isotopy}
Two trisections $(X_1, X_2, X_3)$ and $(Y_1,Y_2,Y_3)$ of a closed 4-manifold $Z$ are \textit{isotopic} if there exists an isotopy $\{h_t\}_{0 \le t \le 1} \colon Z \to Z$ such that $h_0=id_{Z}$ and $h_1(X_i)=Y_i$ for each $i=1,2,3$.
\end{dfn}

\begin{dfn}\label{def:trisection diagram}
A 4-tuple $(\Sigma_g;\alpha,\beta,\gamma)$ is called a $(g;k_1,k_2,k_3)$-\textit{trisection diagram} if the following holds:
\begin{itemize}
\item $(\Sigma_g;\alpha,\beta)$ is a Heegaard diagram of $\#_{k_1}S^1 \times S^2$,
\item $(\Sigma_g;\beta,\gamma)$ is a Heegaard diagram of $\#_{k_2}S^1 \times S^2$ and 
\item $(\Sigma_g;\gamma,\alpha)$ is a Heegaard diagram of $\#_{k_3}S^1 \times S^2$.
\end{itemize}
\end{dfn}

Similar to a trisection, the type of a balanced trisection diagram is also written as $(g,k)$. Cut systems $\alpha$, $\beta$ and $\gamma$ are shown in red, blue and green, respectively. For example, Figure \ref{fig:CP^2} describes a $(1,0)$-trisection diagram of $\mathbb{C}P^2$.

\begin{figure}[h]
\begin{center}
\includegraphics[width=8cm, height=2.5cm, keepaspectratio, scale=1]{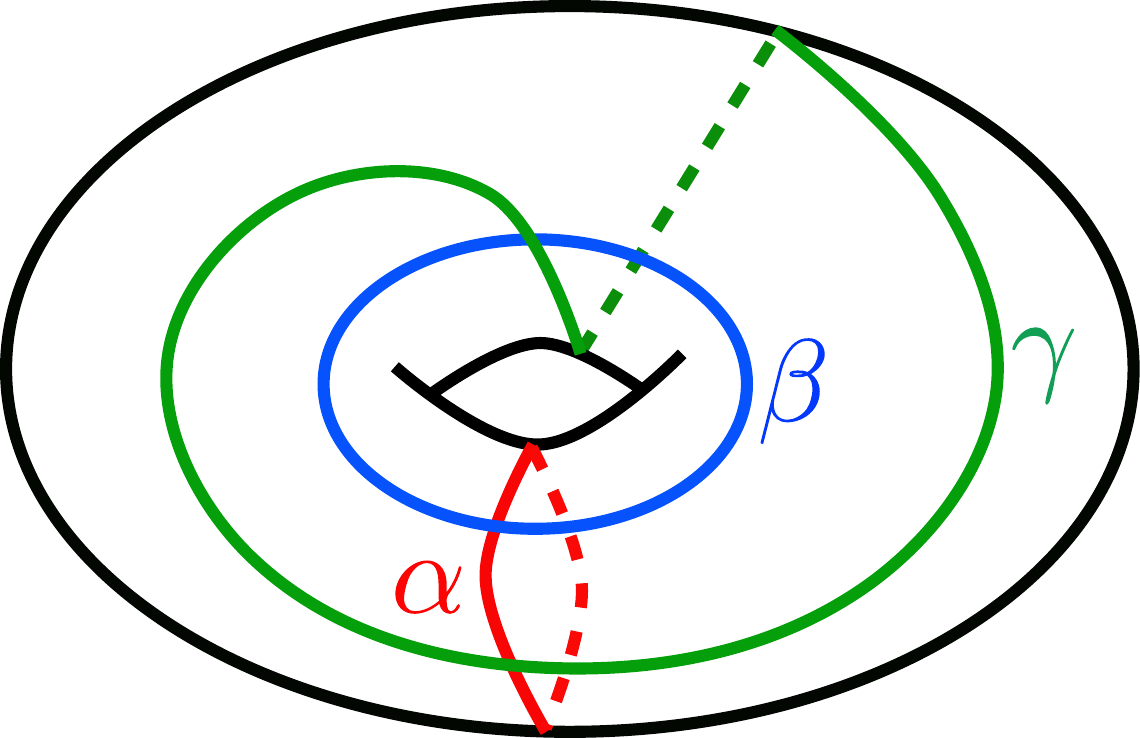}
\end{center}
\setlength{\captionmargin}{50pt}
\caption{A $(1,0)$-trisection diagram of $\mathbb{C}P^2$.}
\label{fig:CP^2}
\end{figure}

\begin{rem}\label{rem:correspond}
Given a genus $g$ trisection whose spine is $H_{\alpha} \cup H_{\beta} \cup H_{\gamma}$, let $\alpha$ (resp. $\beta$, $\gamma$) be the boundary of meridian disk systems of $H_\alpha$ (resp. $H_\beta$, $H_\gamma)$. Then, $(\Sigma_g;\alpha,\beta,\gamma)$ is the trisection diagram with respect to the trisection. Conversely, given a trisection diagram $(\Sigma_g;\alpha,\beta,\gamma)$, by attaching 2-handles to $\Sigma_g \times D^2$ along $\alpha \times \{e^{\frac{2{\pi}i}{3}}\}$, $\beta \times \{e^{\frac{4{\pi}i}{3}}\}$ and $\gamma \times \{e^{2{\pi}i}\}$ with surface framing, we can construct the trisected 4-manifold corresponding to the trisection diagram. Note that in the last process, we have only constructed the spine of the trisection since a trisection is uniquely determined by its spine \cite{LP}.
\end{rem}

\begin{dfn}\label{def:stabili}
Let $(X_1,X_2,X_3)$ be a trisection and $C_{ij}$ a boundary-parallel arc properly embedded in $X_i \cap X_j$. For $\{i,j,k\}=\{1,2,3\}$, we define $X_i^{'}$, $X_j^{'}$ and $X_k^{'}$ as follows:
\begin{itemize}
\item $X_i^{'}=X_i-\nu(C_{ij})$,
\item $X_j^{'}=X_j-\nu(C_{ij})$ and
\item $X_k^{'}=X_k \cup \overline{\nu(C_{ij})}$,
\end{itemize}
where $\nu(C_{ij})$ is a tubular neighborhood of $C_{ij}$. 
The replacement of $(X_1,X_2,X_3)$ by $(X_1^{'},X_2^{'},X_3^{'})$ is called the $k$-\textit{stabilization}. The reverse operation is called the \textit{k-destabilization}.
\end{dfn}

\begin{dfn}
A \textit{stabilization} of a trisection diagram is the connected-sum of the trisection diagram and one of the genus-1 trisection diagram of $S^4$ in Figure \ref{fig:stabilizationfordiagram}, or the trisection diagram itself obtained by the connected-sum. The reverse operation is called a \textit{destabilization}.
\end{dfn}

\begin{figure}[h]
\begin{center}
\includegraphics[width=8cm, height=3cm, keepaspectratio, scale=1]{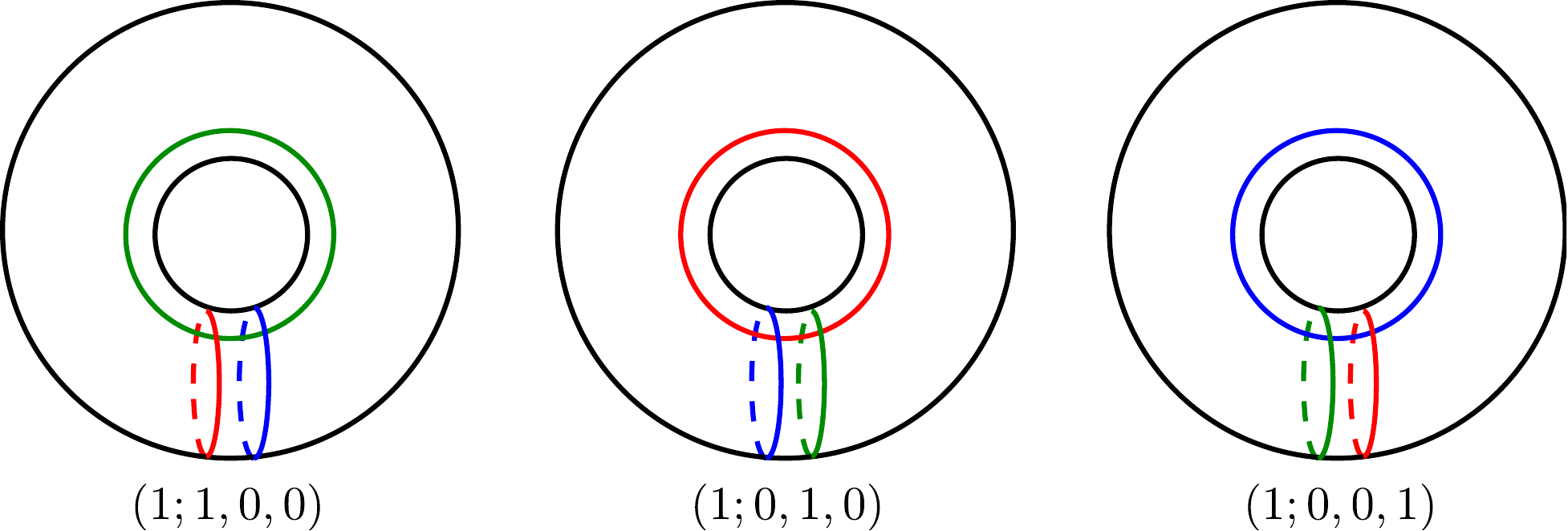}
\end{center}
\setlength{\captionmargin}{50pt}
\caption{Genus-1 trisection diagrams of $S^4$.}
\label{fig:stabilizationfordiagram}
\end{figure}

Note that these (de)stabilizations are the unbalanced one. It is obvious that if we stabilizes a $(g;k_1,k_2,k_3)$-trisection diagram by the $(1;1,0,0)$-trisection diagram (resp. $(1;0,1,0)$- and $(1;0,0,1)$-), the type of the resulting trisection diagram is $(g+1,k_1+1,k_2,k_3)$ (resp. $(g+1,k_1,k_2+1,k_3)$ and $(g+1,k_1,k_2,k_3+1)$). Similarly, the type decreases in the case of destabilizations.

\begin{thm}[{\cite[Theorem 4, Theorem 11]{MR3590351}}]
Every closed 4-manifold admits a trisection. Any two trisections of the same closed 4-manifold are isotopic after performing some number of stabilizations.
\end{thm}

\begin{cor}[{\cite[Corollary 12]{MR3590351}}]\label{cor:henkei}
Two closed 4-manifolds are diffeomorphic if and only if two corresponding trisection diagrams are related by surface diffeomorphisms, handle slides among the same family curves and (de)stabilizations.
\end{cor}

See \cite{MR3590351} for more details.

\subsection{An algorithm converting handle diagrams into trisection diagrams}

Kepplinger \cite{MR4460230} developed an algorithm that takes handle diagrams to trisection diagrams. In this subsection, we recall a simplified version of the algorithm also introduced in \cite{MR4460230}, which can make the genus of the resulting trisection diagram smaller significantly.

\begin{thm}[{\cite[Algorithm 1 and Remark 2.2]{MR4460230}}]\label{thm:algorithm}
Let $X$ be a closed orientable 4-manifold. Suppose that $X$ is described by a handle diagram with some 1-handles and a framed attaching link $L=K_1 \cup \dots \cup K_\ell$. One can obtain a trisection diagram of $X$ from this handle diagram by performing the following steps:

\begin{enumerate}
\item Convert each pair of attaching balls describing a 1-handle into a pair of disks and add a parallel pair of red and blue curves parallel to the boundary of the disk.
\item Add a +1- or (-1)-kink to an attaching circle which has no self-crossings.
\item Convert each crossing as in Figure \ref{fig:convert}. Then, one matches the framing of $K_i$ with the surface framing by winding around a handle (i.e. Dehn twisting on the red curve in Figure \ref{fig:convert}) that only $K_i$ runs. Let $g$ be the number of $\alpha$-, or equivalently, $\beta$-curves at the end of this step.
\item Handle sliding $\alpha$-curves so that $\lvert \alpha_i \cap K_j \rvert = \delta_{ij}$ ($1 \le i \le g$, $1 \le j \le \ell$).
\item Convert $K_j$ to $\gamma_j$.
\item Draw $\gamma_i$ ($\ell+1 \le i \le g$) as a curve parallel to $\alpha_i$ which does not intersect with $K_j$ ($1 \le j \le \ell$).
\end{enumerate}
\end{thm}

\begin{figure}[h]
\begin{center}
\includegraphics[width=5cm, height=8cm, keepaspectratio, scale=1]{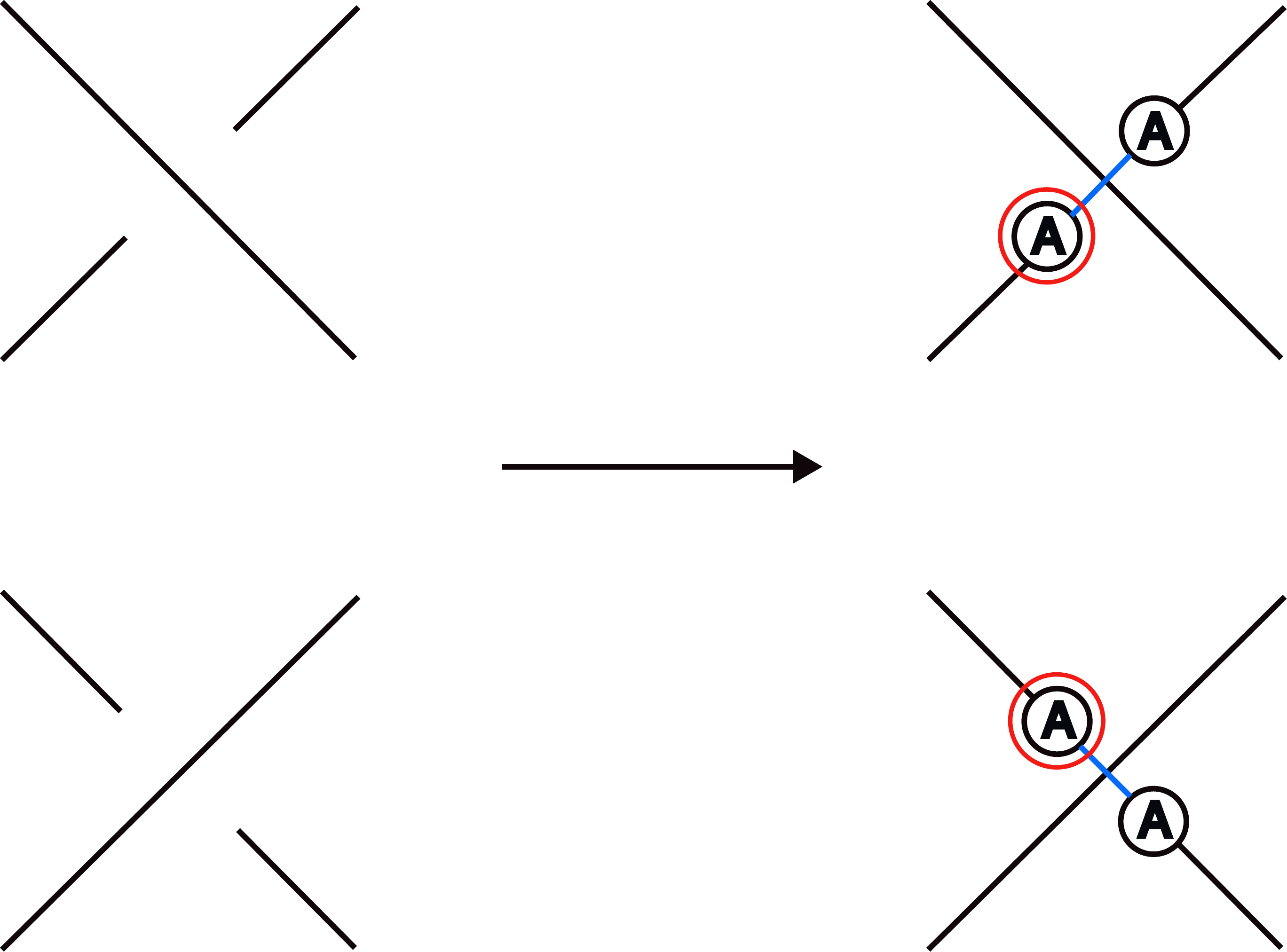}
\end{center}
\setlength{\captionmargin}{65pt}
\caption{At Step 3 in the algorithm, each crossing is converted as described in this figure.}
\label{fig:convert}
\end{figure}

\section{Main theorem}

In this section, we clarify a way to obtain an explicit minimal genus trisection diagram of the elliptic surface $E(n)$. Firstly, we construct a large genus trisection diagram of $E(n)$ from its handle diagram in Figure \ref{fig:E(n)_Kirby}, using the algorithm in Theorem \ref{thm:algorithm}.

\begin{figure}[h]
\begin{center}
\includegraphics[width=9cm, height=5cm, keepaspectratio, scale=1]{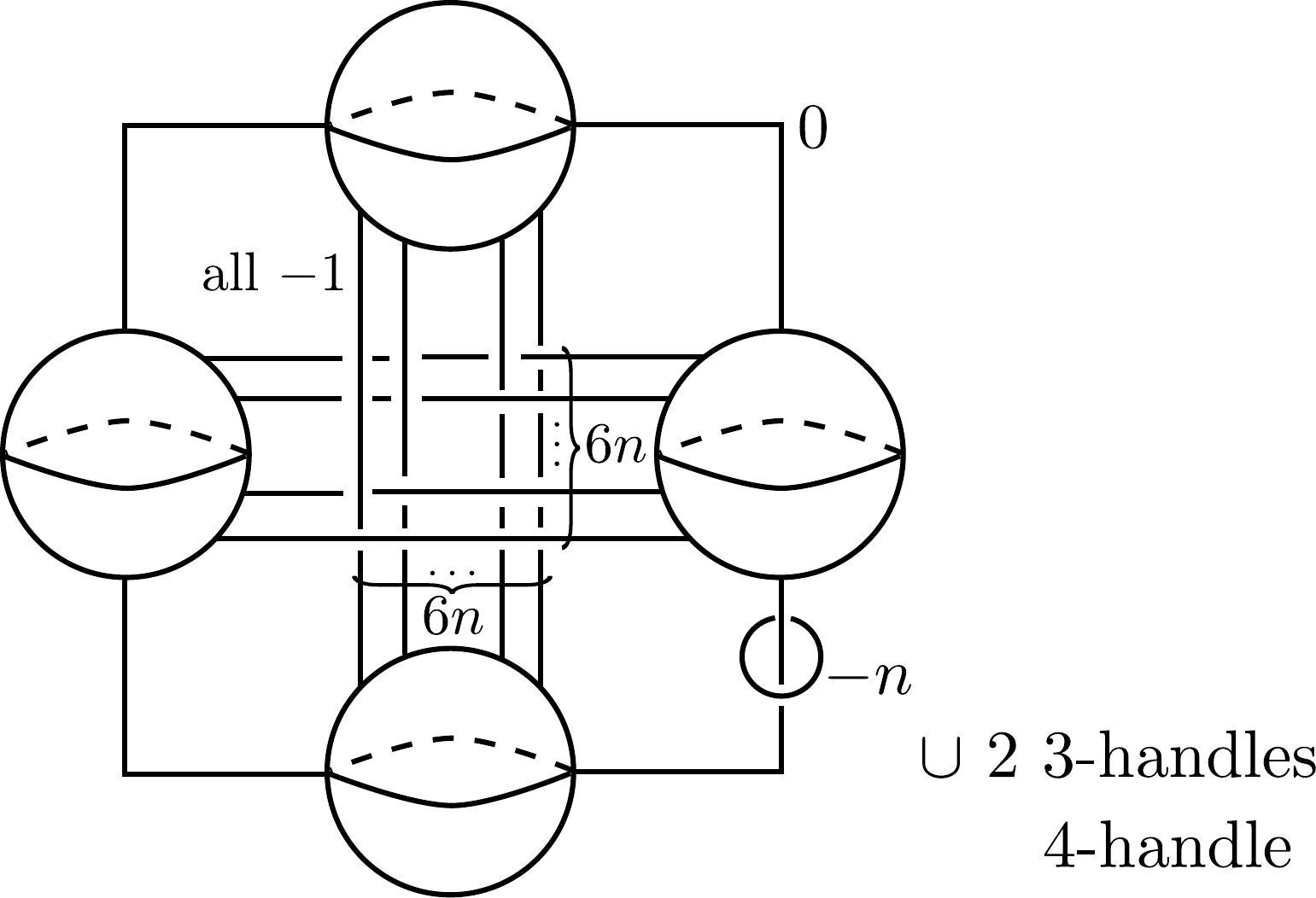}
\end{center}
\setlength{\captionmargin}{50pt}
\caption{A handle diagram of $E(n)$ obtained from its Lefschetz fibration (see \cite[Figure 8.11]{MR1707327}).}
\label{fig:E(n)_Kirby}
\end{figure}

\begin{prop}\label{prop:E(n)}
Figure \ref{fig:step6} is a $(36n^2+6n+6;2,2,36n^2-6n+4)$-trisection diagram of $E(n)$ obtained by performing the algorithm in Theorem \ref{thm:algorithm} to the handle diagram in Figure \ref{fig:E(n)_Kirby}. 
\end{prop}

\begin{proof}

Step 1 in the algorithm convert Figure \ref{fig:E(n)_Kirby} into Figure \ref{fig:step1}. Then, Figure \ref{fig:step2,3} is obtained from Figure \ref{fig:step1} by Steps 2 and 3 in the algorithm. After that, Steps 4 and 5 in the algorithm yields Figure \ref{fig:step4,5}. Finally, we have Figure \ref{fig:step6} from Step 6 in the algorithm, which is a trisection diagram of $E(n)$. Let $(g;k_1,k_2,k_3)$ be the type of this trisection diagram. One can check by Definition \ref{def:trisection diagram} that $g=36n^2+6n+6$, $k_1=2$ and $k_3=36n^2-6n+4$. It is known \cite{MR3590351} that for a closed 4-manifold $X$ admitting a $(g';k_1',k_2',k_3')$-trisection, $\chi(X)=2+g'-(k_1'+k_2'+k_3')$, where $\chi(X)$ is the Euler characteristic of $X$. Since $\chi(E(n))=12n$, we have $k_2=2$.
\end{proof}

\begin{figure}[h]
\begin{center}
\includegraphics[width=9cm, height=5cm, keepaspectratio, scale=1]{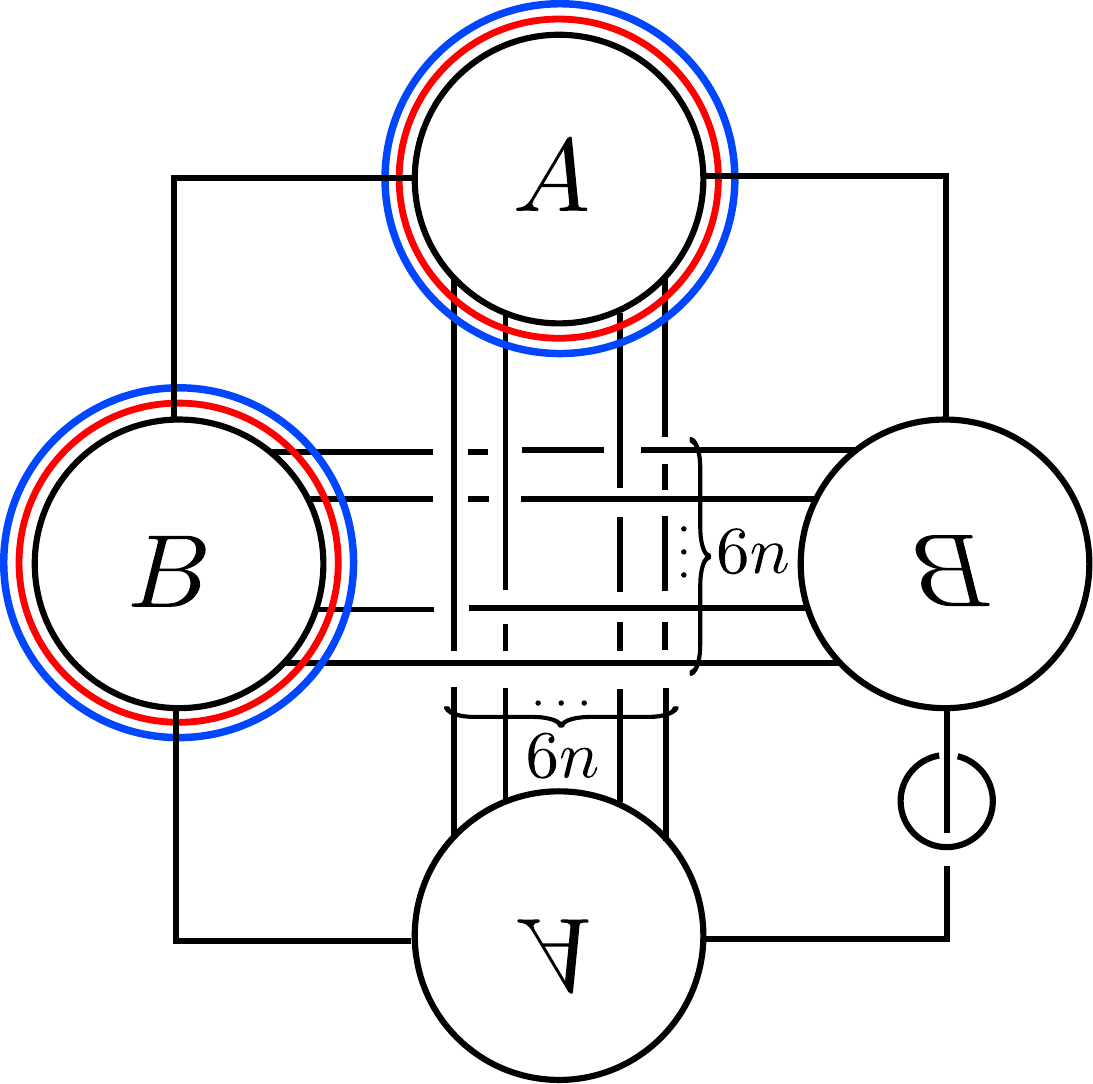}
\end{center}
\setlength{\captionmargin}{50pt}
\caption{After Step1 in the algorithm in Theorem \ref{thm:algorithm}.}
\label{fig:step1}
\end{figure}

\begin{figure}[h]
\begin{center}
\includegraphics[width=12cm, height=15cm, keepaspectratio, scale=1]{step2,3.pdf}
\end{center}
\setlength{\captionmargin}{50pt}
\caption{After Steps 2 and 3 in the algorithm in Theorem \ref{thm:algorithm}. The black $-1$ and $-n$ omit the Dehn twist for $\gamma$-curves as in Figure \ref{fig:notation}.}
\label{fig:step2,3}
\end{figure}

\begin{figure}[h]
\begin{center}
\includegraphics[width=11cm, height=15cm, keepaspectratio, scale=1]{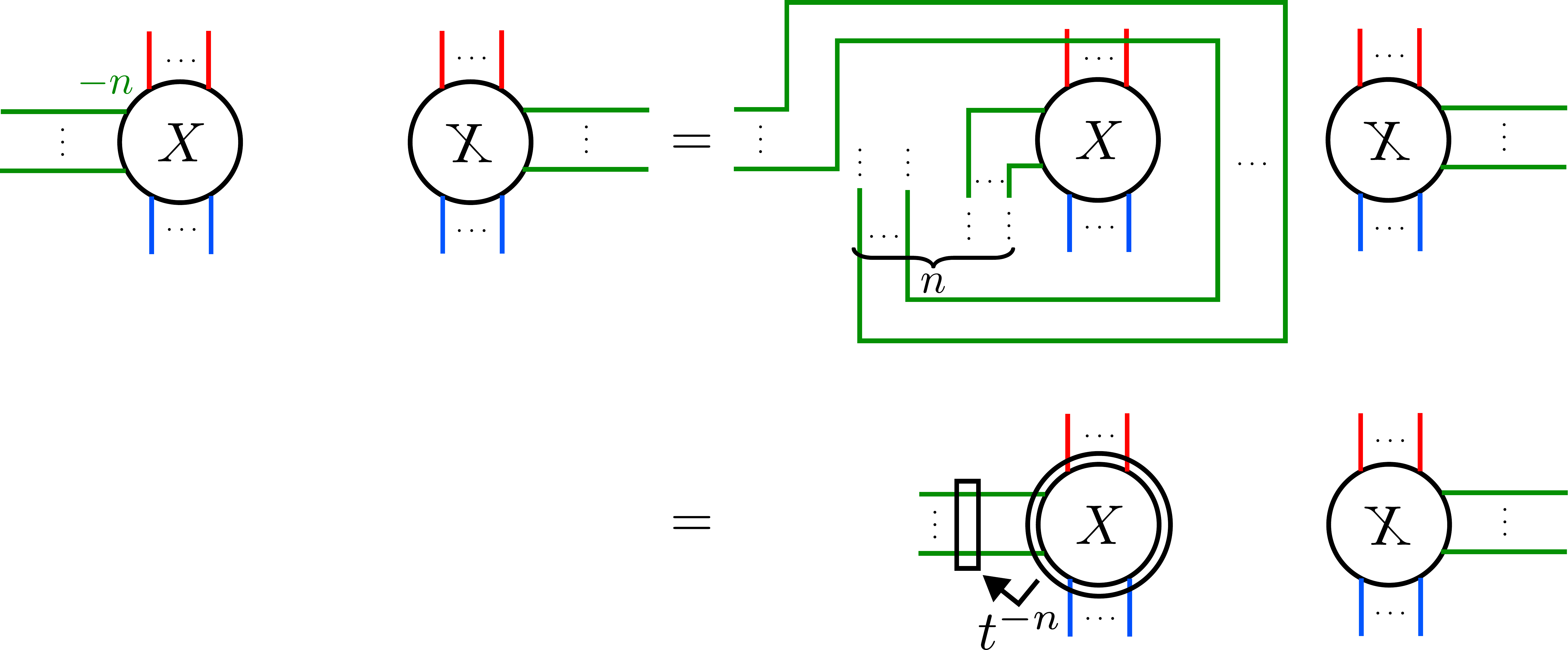}
\end{center}
\setlength{\captionmargin}{45pt}
\caption{Notations in Figures \ref{fig:step2,3} and \ref{fig:god2} for non-negative integers $n$. The bottom figure describes that we perform the Dehn twist $t^{-n}$ along the black curve for all $\gamma$-curves intersecting with the black curve (labeled the rectangle).}
\label{fig:notation}
\end{figure}

\begin{figure}[h]
\begin{center}
\includegraphics[width=12cm, height=15cm, keepaspectratio, scale=1]{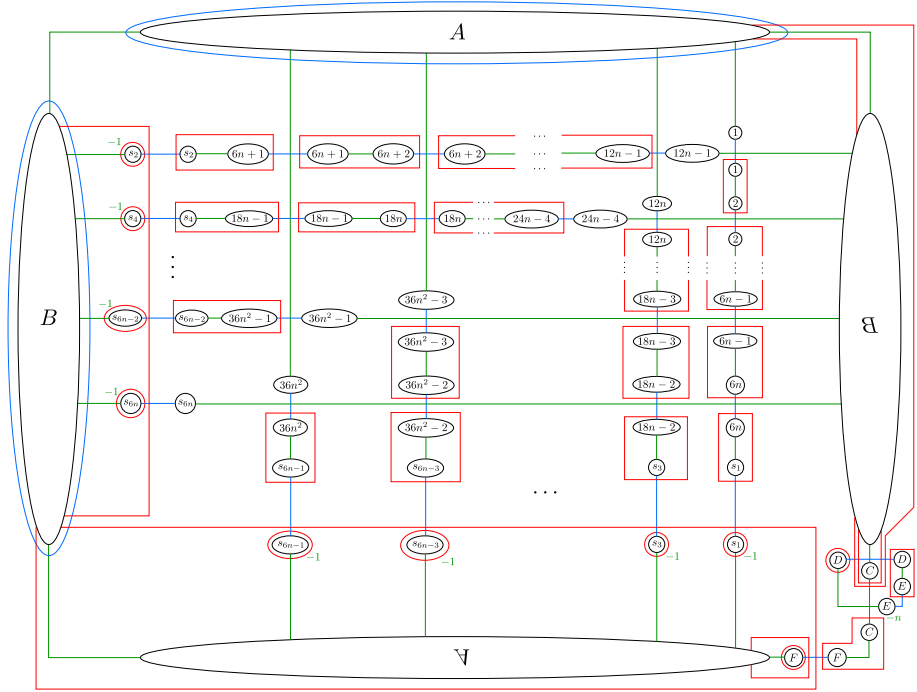}
\end{center}
\setlength{\captionmargin}{50pt}
\caption{After Steps 4 and 5 in the algorithm in Theorem \ref{thm:algorithm}.}
\label{fig:step4,5}
\end{figure}

\begin{figure}[h]
\begin{center}
\includegraphics[width=12cm, height=15cm, keepaspectratio, scale=1]{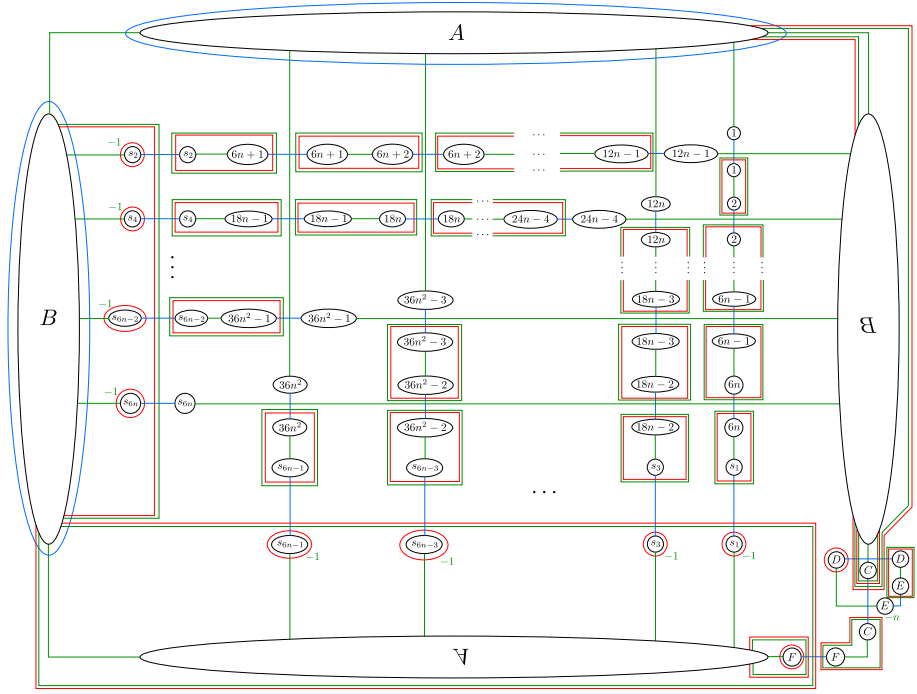}
\end{center}
\setlength{\captionmargin}{35pt}
\caption{After Step 6 in the algorithm in Theorem \ref{thm:algorithm}. This is a $(36n^2+6n+6;2,2,36n^2-6n+4)$-trisection diagram of $E(n)$.}
\label{fig:step6}
\end{figure}

Here is the main theorem mentioned in Section \ref{sec:intro}.

\begin{thm}\label{thm:main theorem}
A $(12n-2,0)$-trisection diagram of $E(n)$ can be obtained from a $(36n^2+6n+6;2,2,36n^2-6n+4)$-trisection diagram of $E(n)$ in Figure \ref{fig:step6} by handle slides and destabilizations. 
\end{thm}

We prove Theorem \ref{thm:main theorem} by destabilizing the trisection diagram in Figure \ref{fig:step6} and reducing its genus. Note that handle slides and destabilizations do not change the diffeomorphism type of the 4-manifold (Corollary \ref{cor:henkei}). Destabilizations in the following proofs are based on \cite[Lemma 8]{MR4480889}. 

\begin{lem}\label{lem:destab1}
A $(12n+4,2)$-trisection diagram of $E(n)$ shown in Figure \ref{fig:destab1} is obtained from Figure \ref{fig:step6} by handle slides and destabilizations.
\end{lem}

\begin{proof}

In Figure \ref{fig:step6}, we can take many destabilizations for the parallel $\alpha$- and $\gamma$-curves, and $\beta$-curves in the center of the figure as in Figure \ref{fig:destabilization}. By these destabilizations, we have Figure \ref{fig:destab1}, which is a $(12n+4,2)$-trisection diagram of $E(n)$. Note that we rename the labels of the disks in Figure \ref{fig:destab1}.
\end{proof}

\begin{figure}[h]
\begin{center}
\includegraphics[width=12cm, height=15cm, keepaspectratio, scale=1]{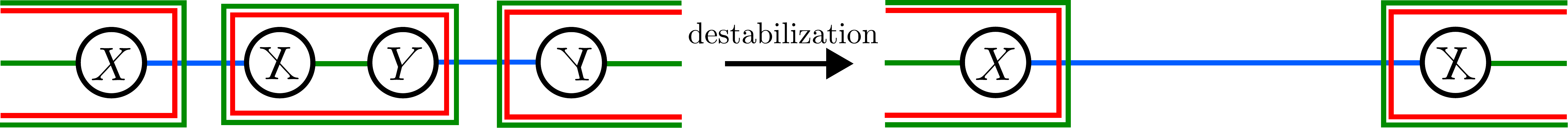}
\end{center}
\setlength{\captionmargin}{43pt}
\caption{A destabilization performed from Figure \ref{fig:step6} to Figure \ref{fig:destab1}.}
\label{fig:destabilization}
\end{figure}

\begin{figure}[h]
\begin{center}
\includegraphics[width=12cm, height=15cm, keepaspectratio, scale=1]{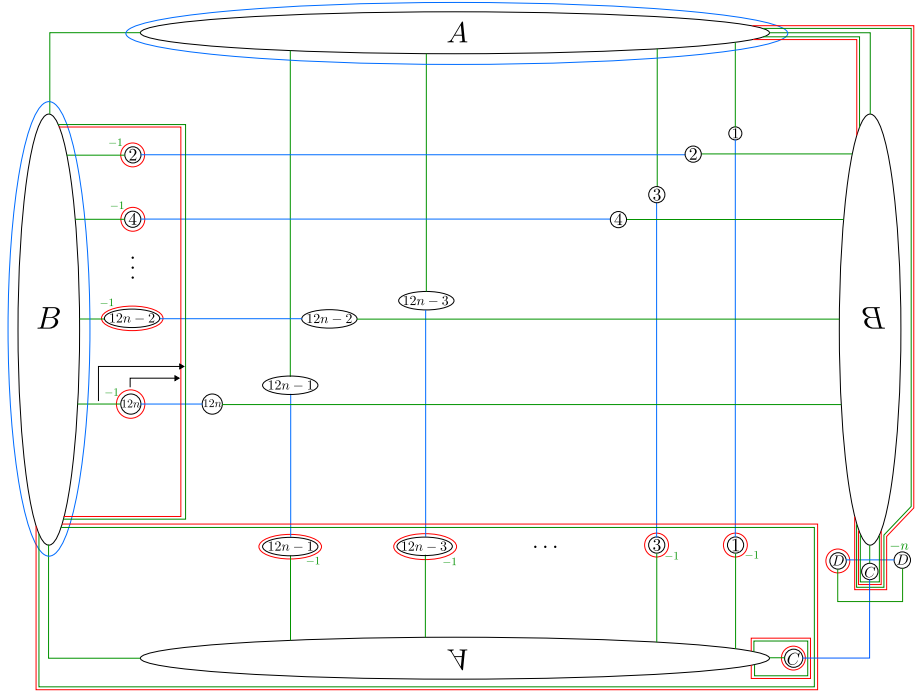}
\end{center}
\setlength{\captionmargin}{50pt}
\caption{A $(12n+4,2)$-trisection diagram of $E(n)$ obtained from Figure \ref{fig:step6} by destabilizations.}
\label{fig:destab1}
\end{figure}

In the following proofs, we indicate the choice of handle slides by arrows as in Figure \ref{fig:handle slide2}.

\begin{figure}[h]
\begin{center}
\includegraphics[width=8cm, height=15cm, keepaspectratio, scale=1]{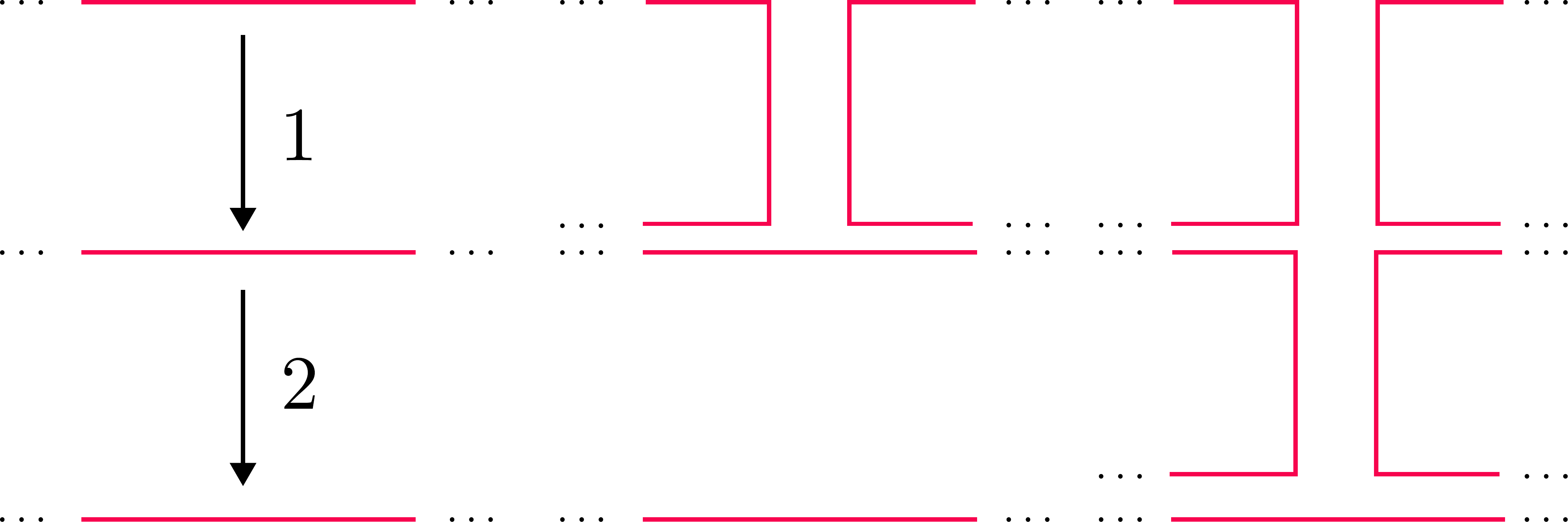}
\end{center}
\setlength{\captionmargin}{50pt}
\caption{In this paper, handle slides are indicated by arrows as in this figure. If it is necessary to specify the order of handle slides, we represent them by labeling arrows as in this figure. Namely, the most left figure means that firstly we perform the handle slide labeled 1 (the central figure), and then we perform the handle slide labeled 2 (the most right figure).}
\label{fig:handle slide2}
\end{figure}

\begin{lem}\label{lem:destab2}
A $(12n+2;2,2,0)$-trisection diagram of $E(n)$ shown in Figure \ref{fig:after_destab2} is obtained from Figure \ref{fig:destab1} by handle slides and destabilizations.
\end{lem}

\begin{proof}

In Figure \ref{fig:destab1}, by sliding $\alpha$- and $\gamma$-curves as arrows in Figure \ref{fig:destab1}, we have Figure \ref{fig:before_destab1}. Then, by destabilizing the disk labeled $12n$ in Figure \ref{fig:before_destab1}, Figure \ref{fig:after_destab1} is obtained. After that, in Figure \ref{fig:after_destab1}, by sliding $\alpha$- and $\gamma$-curves as arrows in Figure \ref{fig:after_destab1}, we have Figure \ref{fig:before_destab2}. Then, by destabilizing the disk labeled $12n-1$ in Figure \ref{fig:before_destab2}, Figure \ref{fig:after_destab2} is obtained.
\end{proof}

\begin{figure}[h]
\begin{center}
\includegraphics[width=12cm, height=15cm, keepaspectratio, scale=1]{before_destab1.pdf}
\end{center}
\setlength{\captionmargin}{50pt}
\caption{A $(12n+4,2)$-trisection diagram of $E(n)$ obtained from Figure \ref{fig:destab1} by the handle slides.}
\label{fig:before_destab1}
\end{figure}

\begin{figure}[h]
\begin{center}
\includegraphics[width=12cm, height=15cm, keepaspectratio, scale=1]{after_destab1.pdf}
\end{center}
\setlength{\captionmargin}{50pt}
\caption{A $(12n+3;2,2,1)$-trisection diagram of $E(n)$ obtained from Figure \ref{fig:before_destab1} by the destabilization.}
\label{fig:after_destab1}
\end{figure}

\begin{figure}[h]
\begin{center}
\includegraphics[width=12cm, height=15cm, keepaspectratio, scale=1]{before_destab2.pdf}
\end{center}
\setlength{\captionmargin}{50pt}
\caption{A $(12n+3;2,2,1)$-trisection diagram of $E(n)$ obtained from Figure \ref{fig:after_destab1} by the handle slides.}
\label{fig:before_destab2}
\end{figure}

\begin{figure}[h]
\begin{center}
\includegraphics[width=12cm, height=15cm, keepaspectratio, scale=1]{after_destab2.pdf}
\end{center}
\setlength{\captionmargin}{50pt}
\caption{A $(12n+2;2,2,0)$-trisection diagram of $E(n)$ obtained from Figure \ref{fig:before_destab2} by the destabilization.}
\label{fig:after_destab2}
\end{figure}

\begin{lem}\label{lem:destab3}
A $(12n;0,2,0)$-trisection diagram of $E(n)$ shown in Figure \ref{fig:after_destab4} is obtained from Figure \ref{fig:after_destab2} by handle slides and destabilizations.
\end{lem}

\begin{proof}

In Figure \ref{fig:after_destab2}, by sliding the two complicated $\alpha$-curves over each $\alpha$-curve on the disks labeled 1, $\ldots$, $12n-2$, $C$ and $D$ in Figure \ref{fig:after_destab2}, we have Figure \ref{fig:before_destab3.1}. Then, by sliding $\gamma$-curves as arrows in Figure \ref{fig:before_destab3.1}, we obtain Figure \ref{fig:before_destab3.2}. After that, by destabilizing the disk labeled $B$ in Figure \ref{fig:before_destab3.2}, Figure \ref{fig:after_destab3} is obtained. In Figure \ref{fig:after_destab3}, by sliding $\gamma$-curves as arrows in Figure \ref{fig:after_destab3}, we have Figure \ref{fig:before_destab4}. Then, by destabilizing the disk labeled $A$ in Figure \ref{fig:before_destab4}, Figure \ref{fig:after_destab4} is obtained.
\end{proof}

\begin{figure}[h]
\begin{center}
\includegraphics[width=11cm, height=15cm, keepaspectratio, scale=1]{before_destab3.1.pdf}
\end{center}
\setlength{\captionmargin}{50pt}
\caption{A $(12n+2;2,2,0)$-trisection diagram of $E(n)$ obtained from Figure \ref{fig:after_destab2} by the handle slides.}
\label{fig:before_destab3.1}
\end{figure}

\begin{figure}[h]
\begin{center}
\includegraphics[width=11cm, height=15cm, keepaspectratio, scale=1]{before_destab3.2.pdf}
\end{center}
\setlength{\captionmargin}{50pt}
\caption{A $(12n+2;2,2,0)$-trisection diagram of $E(n)$ obtained from Figure \ref{fig:before_destab3.1} by the handle slides.}
\label{fig:before_destab3.2}
\end{figure}

\begin{figure}[h]
\begin{center}
\includegraphics[width=10cm, height=15cm, keepaspectratio, scale=1]{after_destab3.pdf}
\end{center}
\setlength{\captionmargin}{50pt}
\caption{A $(12n+1;1,2,0)$-trisection diagram of $E(n)$ obtained from Figure \ref{fig:before_destab3.2} by the destabilization.}
\label{fig:after_destab3}
\end{figure}

\begin{figure}[h]
\begin{center}
\includegraphics[width=10cm, height=15cm, keepaspectratio, scale=1]{before_destab4.pdf}
\end{center}
\setlength{\captionmargin}{50pt}
\caption{A $(12n+1;1,2,0)$-trisection diagram of $E(n)$ obtained from Figure \ref{fig:after_destab3} by the handle slides.}
\label{fig:before_destab4}
\end{figure}

\begin{figure}[h]
\begin{center}
\includegraphics[width=11cm, height=15cm, keepaspectratio, scale=1]{after_destab4.pdf}
\end{center}
\setlength{\captionmargin}{50pt}
\caption{A $(12n;0,2,0)$-trisection diagram of $E(n)$ obtained from Figure \ref{fig:before_destab4} by the destabilization.}
\label{fig:after_destab4}
\end{figure}

We consider the following lemma before proving Theorem \ref{thm:main theorem}.

\begin{lem}\label{lem:destabilization2}
Let $(\Sigma_g; \beta, \gamma)$ be a Heegaard diagram, where $\beta=\{\beta_1, \ldots, \beta_g\}$ and $\gamma=\{\gamma_1, \ldots, \gamma_g\}$. Suppose that $\beta_i$ is in standard position as in Figure \ref{fig:destabilization2}, intersects with $\gamma_i$ once and does not intersect with other $\gamma$-curves for some $i=1, \ldots, g$. Let $(\Sigma'_{g-1}; \beta', \gamma')$ be a Heegaard diagram obtained by destabilizing $(\Sigma_g; \beta, \gamma)$ for $\beta_i$ and $\gamma_i$, where $\beta'=\{\beta'_1, \ldots, \beta'_{g-1}\}$ and $\gamma'=\{\gamma'_1, \ldots, \gamma'_{g-1}\}$. Then, if $\beta'_1, \ldots, \beta'_n$ are parallel to $\gamma'_1, \ldots, \gamma'_n$ on $\Sigma'_{g-1}$, $\beta_1, \ldots, \beta_n$ can be parallel to $\gamma_1, \ldots, \gamma_n$ on $\Sigma_g$ by handle slides over $\beta_i$.
\end{lem}

\begin{proof}
In Figure \ref{fig:destabilization2}, some $\beta$-curves can be parallel to $\gamma_{j+1}, \ldots, \gamma_{n}$ by handle sliding them over $\beta_i$ once for each strand. The other $\beta$-curves can be parallel to $\gamma_{1}, \ldots, \gamma_{j}$ by handle sliding them over $\beta_i$ twice for each strand as in Figure \ref{fig:handle slide}. 
\end{proof}

\begin{figure}[h]
\begin{center}
\includegraphics[width=8cm, height=15cm, keepaspectratio, scale=1]{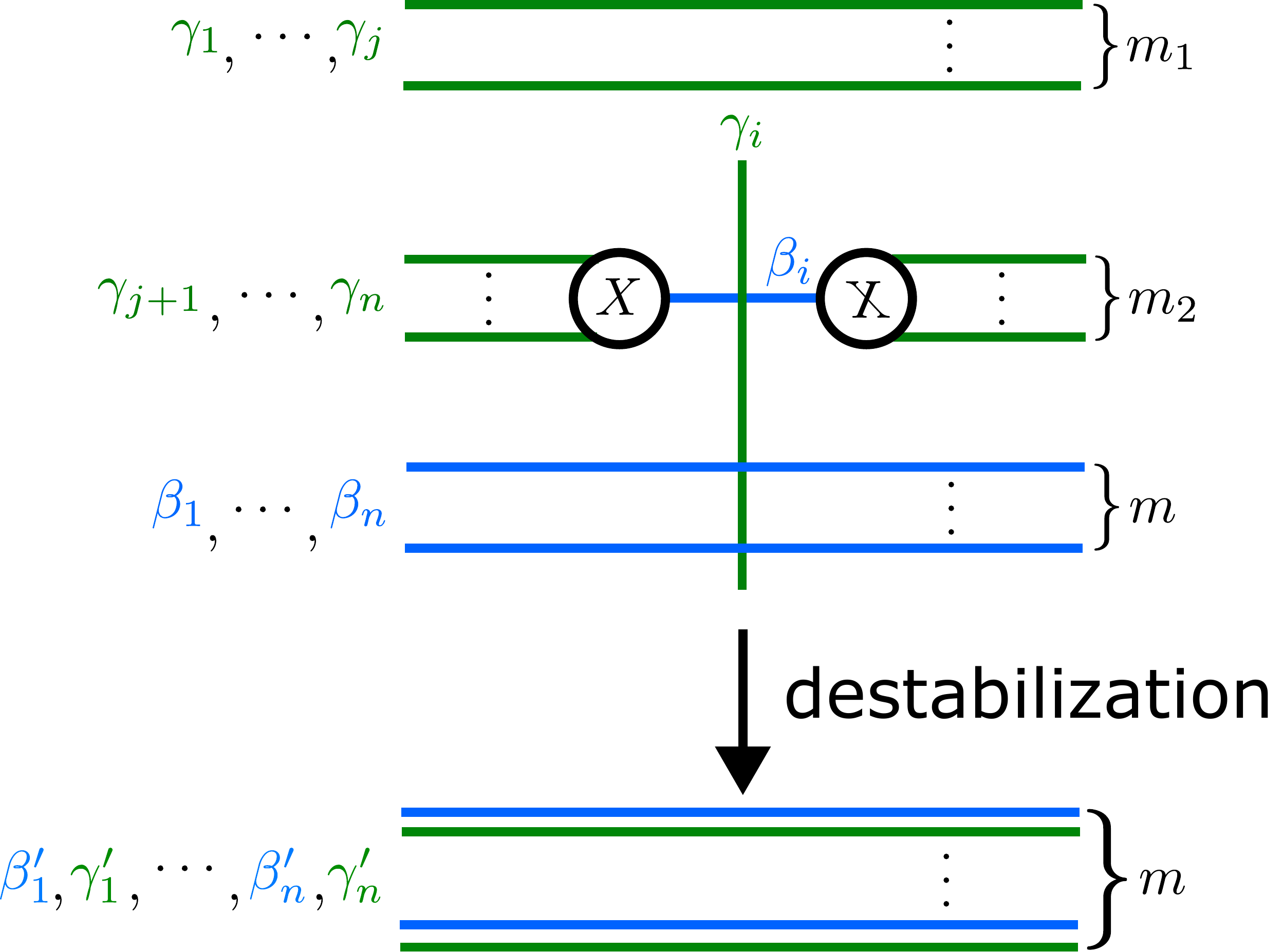}
\end{center}
\setlength{\captionmargin}{45pt}
\caption{A destabilization for a Heegaard diagram. If $\beta'_1, \ldots, \beta'_n$ are parallel to $\gamma'_1, \ldots, \gamma'_n$, then $\beta_1, \ldots, \beta_n$ can be parallel to $\gamma_1, \ldots, \gamma_n$ by handle slides over $\beta_i$.}
\label{fig:destabilization2}
\end{figure}

\begin{figure}[h]
\begin{center}
\includegraphics[width=7cm, height=15cm, keepaspectratio, scale=1]{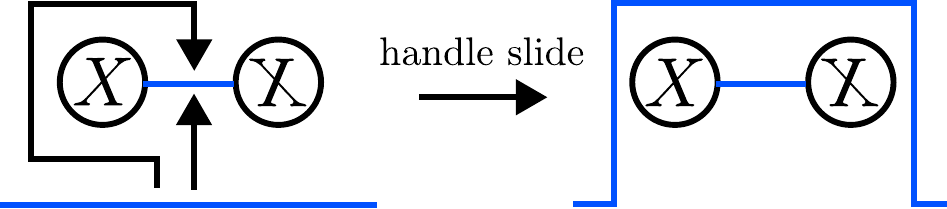}
\end{center}
\setlength{\captionmargin}{50pt}
\caption{Handle slides for $\beta$-curves in Lemma \ref{lem:destabilization2}.}
\label{fig:handle slide}
\end{figure}

\begin{proof}[Proof of Theorem \ref{thm:main theorem}]

Combining Lemmas \ref{lem:destab1}, \ref{lem:destab2} and \ref{lem:destab3}, we have the $(12n;0,2,0)$-trisection diagram of $E(n)$ in Figure \ref{fig:after_destab4} from the $(36n^2+6n+6;2,2,36n^2-6n+4)$-trisection diagram of $E(n)$ in Figure \ref{fig:step6} by handle slides and destabilizations. Thus, it suffices to show that a $(12n-2,0)$-trisection diagram of $E(n)$ can be obtained from Figure \ref{fig:after_destab4} by handle slides and destabilizations.

In Figure \ref{fig:after_destab4}, by sliding three $\gamma$-strands labeled the rectangle over $\gamma_1$, $\gamma_3$, $\ldots$, $\gamma_{12n-5}$ each twice as arrows in Figure \ref{fig:after_destab4}, we have Figure \ref{fig:god1.1}. 
In Figure \ref{fig:god1.1}, by sliding three $\gamma$-strands labeled the rectangle over $\gamma_2$, $\gamma_4$, $\ldots$, $\gamma_{12n-4}$ each twice as arrows in Figure \ref{fig:god1.1}, we have Figure \ref{fig:god1.2}. In Figure \ref{fig:god1.2}, by sliding $\beta_1$, $\beta_3$, $\ldots$, $\beta_{12n-3}$ over $\beta_C$ and $\beta_D$ each twice as in Figure \ref{fig:handle slide}, we have Figure \ref{fig:god2}. 
In Figure \ref{fig:god2}, by sliding $\gamma$-curves and black curve labeled $t^{-n}$ as arrows in Figure \ref{fig:god2}, we have Figure \ref{fig:last}. 
See \cite[Lemma 3.4]{MR5029882} for the slide of the black curve labeled $t^{-n}$. 
In Figure \ref{fig:last}, by sliding $\gamma$-curves as arrows in Figure \ref{fig:last} in order, we have Figure \ref{fig:last2}. In Figure \ref{fig:last2}, by sliding $\gamma$-curves as arrows in Figure \ref{fig:last2}, we have Figure \ref{fig:last3}.

We now focus on the Heegaard diagram $(\beta, \gamma)$ in Figure \ref{fig:last3}. Note that unless otherwise stated, we do not rename the curves even after performing handle slides or destabilizations. 
We can destabilize $(\beta, \gamma)$ for $\beta_1$ and $\gamma_1$ since $\beta_1$ intersects with only $\gamma_1$ once. 
In the destabilized Heegaard diagram, since $\beta_2$ intersects with only $\gamma_2$ once, we can destabilize the Heegaard diagram again. 
By performing such destabilizations for $\beta$- and $\gamma$-curves labeled $3,4, \ldots, 12n-4$, $C$ and $D$, we can obtain a genus-2 Heegaard diagram $(\beta, \gamma)$ of $\#_2 S^1 \times S^2$, where $\beta=\{\beta_{12n-3}, \beta_{12n-2}\}$ and $\gamma=\{\gamma^s_\ast, \gamma^s_{\ast \ast}\}$.  
Here, the $\gamma$-curves obtained by handle sliding $\gamma_\ast$ and $\gamma_{\ast \ast}$ over $\gamma_3, \ldots, \gamma_{12n-4}$ so that each $\beta_i$ intersects with only $\gamma_i$ are denoted by $\gamma^s_\ast$ and $\gamma^s_{\ast \ast}$.
Since genus $g$ Heegaard splittings of $\#_n S^1 \times S^2$ are unique up to isotopy for each $g \ge n$ \cite{MR227992}, $\beta=\{\beta_{12n-3}, \beta_{12n-2}\}$ can be parallel to $\gamma=\{\gamma^s_\ast, \gamma^s_{\ast \ast}\}$ by just handle slides in the genus-2 Heegaard diagram. 
Thus, by Lemma \ref{lem:destabilization2}, $\beta_{12n-3}$ and $\beta_{12n-2}$ also can be parallel to $\gamma^s_\ast$ and $\gamma^s_{\ast \ast}$ by just handle slides in Figure \ref{fig:last3}. 

Since $\alpha_1$ and $\alpha_2$ in Figure \ref{fig:last3} intersects with $\gamma^s_\ast$ and $\gamma^s_{\ast \ast}$ once, respectively, we can destabilize Figure \ref{fig:last3} for firstly $\alpha_1$ and parallel curves $\beta_{12n-3}$ and $\gamma^s_\ast$, and then $\alpha_2$ and parallel curves $\beta_{12n-2}$ and $\gamma^s_{\ast \ast}$. 
This completes the proof.
\end{proof}

\begin{figure}[h]
\begin{center}
\includegraphics[width=11cm, height=15cm, keepaspectratio, scale=1]{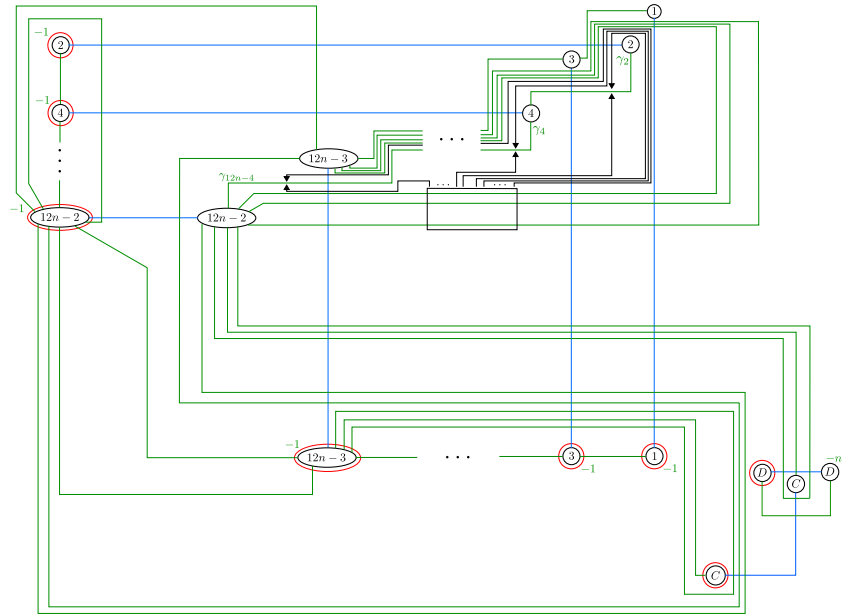}
\end{center}
\setlength{\captionmargin}{50pt}
\caption{A $(12n;0,2,0)$-trisection diagram of $E(n)$ obtained from Figure \ref{fig:after_destab4} by the handle slides.}
\label{fig:god1.1}
\end{figure}

\begin{figure}[h]
\begin{center}
\includegraphics[width=11cm, height=15cm, keepaspectratio, scale=1]{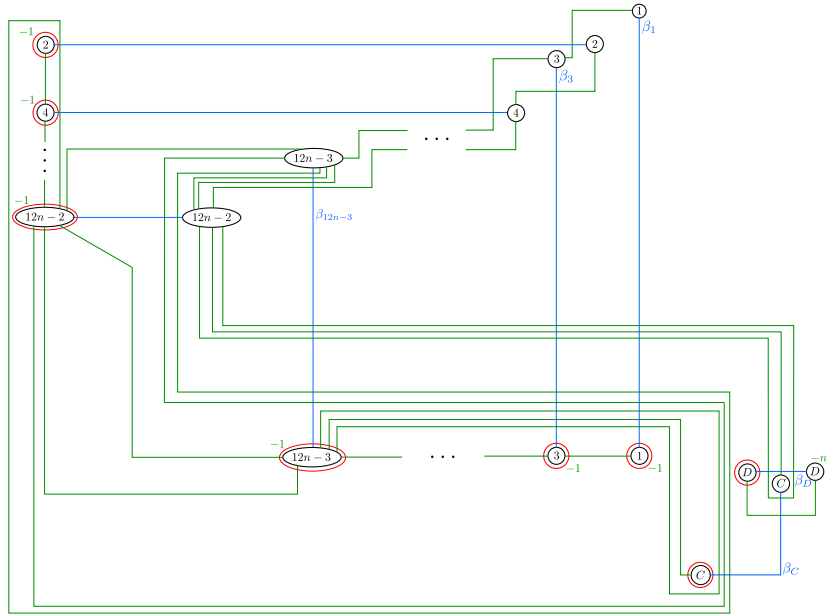}
\end{center}
\setlength{\captionmargin}{50pt}
\caption{A $(12n;0,2,0)$-trisection diagram of $E(n)$ obtained from Figure \ref{fig:god1.1} by the handle slides.}
\label{fig:god1.2}
\end{figure}

\begin{figure}[h]
\begin{center}
\includegraphics[width=10.3cm, height=15cm, keepaspectratio, scale=1]{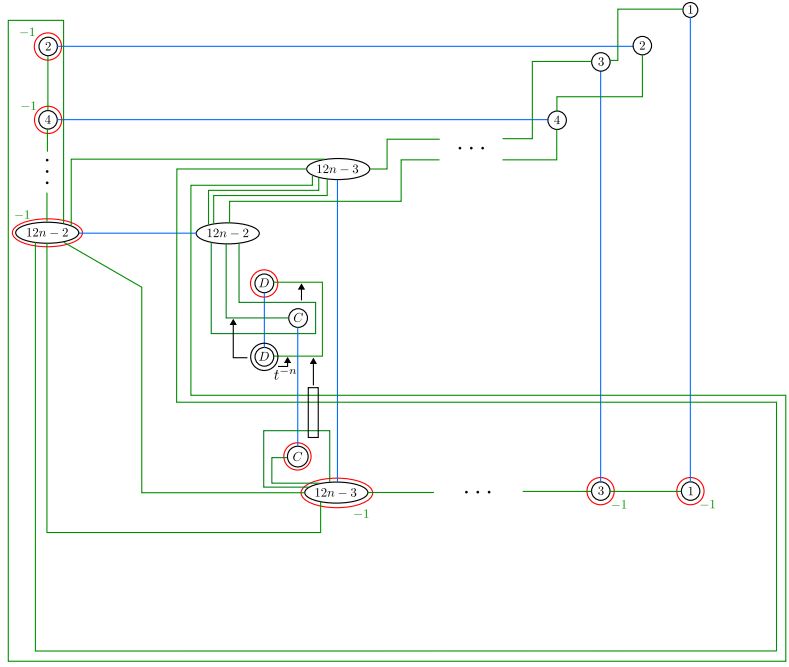}
\end{center}
\setlength{\captionmargin}{20pt}
\caption{A $(12n;0,2,0)$-trisection diagram of $E(n)$ obtained from Figure \ref{fig:god1.2} by the handle slides. See Figure \ref{fig:notation} for the black arrow labeled $t^{-n}$.}
\label{fig:god2}
\end{figure}

\begin{figure}[h]
\begin{center}
\includegraphics[width=10.3cm, height=15cm, keepaspectratio, scale=1]{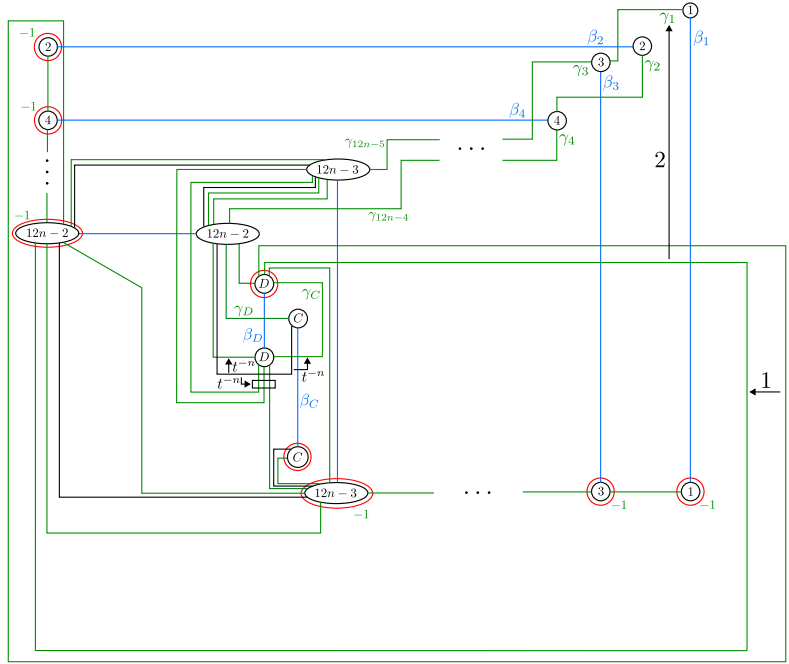}
\end{center}
\setlength{\captionmargin}{50pt}
\caption{A $(12n;0,2,0)$-trisection diagram of $E(n)$ obtained from Figure \ref{fig:god2} by the handle slides.}
\label{fig:last}
\end{figure}

\begin{figure}[h]
\begin{center}
\includegraphics[width=11cm, height=15cm, keepaspectratio, scale=1]{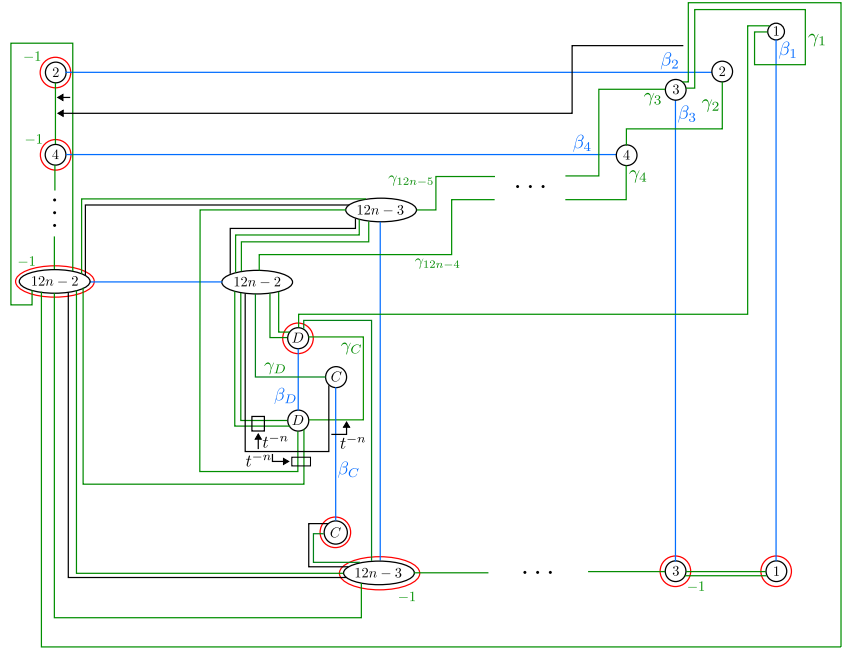}
\end{center}
\setlength{\captionmargin}{50pt}
\caption{A $(12n;0,2,0)$-trisection diagram of $E(n)$ obtained from Figure \ref{fig:last} by the handle slides.}
\label{fig:last2}
\end{figure}

\begin{figure}[h]
\begin{center}
\includegraphics[width=11cm, height=15cm, keepaspectratio, scale=1]{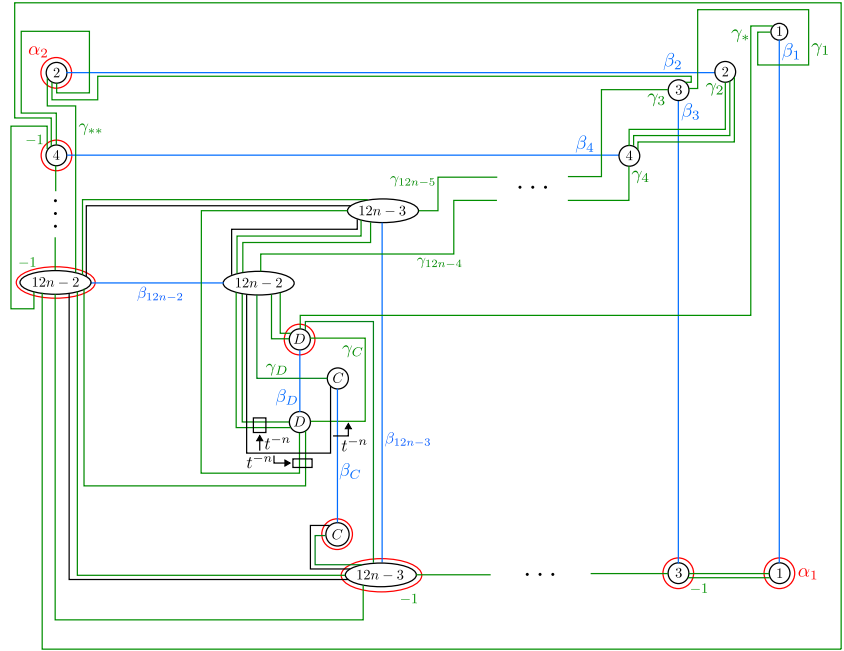}
\end{center}
\setlength{\captionmargin}{50pt}
\caption{A $(12n;0,2,0)$-trisection diagram of $E(n)$ obtained from Figure \ref{fig:last2} by the handle slides.}
\label{fig:last3}
\end{figure}

\begin{rem}
In the process of obtaining the minimal genus trisection diagram of $E(n)$ from Figure \ref{fig:step6}, we do not use surface diffeomorphisms and stabilizations (see Corollary \ref{cor:henkei}).
\end{rem}

From Theorem \ref{thm:main theorem}, we have the following corollary (see \cite[Theorem 7.7]{MR4493476}).

\begin{cor}\label{cor:genus}
The trisection genus of $E(n)$ is $12n-2$.
\end{cor}

We can obtain the minimal genus trisection of $E(n)$ corresponding to the minimal genus trisection diagram of $E(n)$ in Theorem \ref{thm:main theorem} in the sense of Remark \ref{rem:correspond}.

\begin{que}\label{que:isotopic}
Are the minimal genus trisection of $E(n)$ in Theorem \ref{thm:main theorem} and the one in \cite[Theorem 7.7]{MR4493476}, isotopic?
\end{que}

\bibliographystyle{amsalpha}
\bibliography{trisection, GS.Wald, math}

\end{document}